\documentclass[12pt]{article}

\usepackage{latexsym}
\usepackage{epsfig,graphics}
\usepackage{amsmath,amssymb,amsthm}
\usepackage{graphics,epsfig,calc}
\usepackage{upref}
 \newcommand{\beqn}{\begin{eqnarray}}
 \newcommand{\eeqn}{\end{eqnarray}}
 \newcommand{\be}{\begin{equation}}
 \newcommand{\ee}{\end{equation}}
 \newcommand{\ba}{\begin{array}}
 \newcommand{\ea}{\end{array}}
\newcommand{\br}{\begin{remark}}
 \newcommand{\er}{\end{remark}}
\newcommand{\bc}{\begin{cor}}
 \newcommand{\ec}{\end{cor}}

 \newcommand{\pa}{\partial}
 \newcommand{\re}{\ref}
 \newcommand{\ci}{\cite}
 \newcommand{\ds}{\displaystyle}
 \newcommand{\la}{\label}
 \newcommand{\rIm}{{\rm Im\5}}
 
 \newcommand{\supp}{{\rm supp~}}

\newcommand{\fr}{\frac}

\newcommand{\ov}{\overline}

\newcommand{\ti}{\tilde}

\newcommand{\cF}{{\cal F}}

\newcommand{\cH}{{\cal H}}

\newcommand{\cL}{{\cal L}}
\newcommand{\cN}{{\cal N}}
\newcommand{\cM}{{\cal M}}

\newcommand{\cO}{{\cal O}}
\newcommand{\cP}{{\cal P}}
\newcommand{\cR}{{\cal R}}
\newcommand{\cW}{{\cal W}}
\newcommand{\cU}{{\cal U}}
\newcommand{\cV}{{\cal V}}

\newcommand{\ve}{\varepsilon}

\newcommand{\De}{\Delta}
\newcommand{\de}{\delta}

\newcommand{\al}{\alpha}

\newcommand{\Ga}{\Gamma}
\newcommand{\si}{\sigma}
\newcommand{\om}{\omega}

\newcommand{\na}{\nabla}
\newcommand{\Si}{\Sigma}
\newcommand{\lam}{\lambda}

\newcommand{\ka}{\kappa}

\newcommand{\5}{{\hspace{0.5mm}}}

\newcommand{\R}{\mathbb{R}}

\newcommand{\C}{\mathbb{C}}

\renewcommand{\theequation}{\thesection.\arabic{equation}}
\renewcommand{\thesection}{\arabic{section}}
\renewcommand{\thesubsection}{\arabic{section}.\arabic{subsection}}
\newtheorem{theorem}{Theorem}[section]

\renewcommand{\thetheorem}{\arabic{section}.\arabic{theorem}}
\newtheorem{defin}[theorem]{Definition}

\newtheorem{lemma}[theorem]{Lemma}
\newtheorem{remark}[theorem]{Remark}

\newtheorem{cor}[theorem]{Corollary}
\newtheorem{pro}[theorem]{Proposition}
\newcommand{\bp}{\begin{pro}}
\newcommand{\ep}{\end{pro}}

\begin{document}
\begin{titlepage}
\bigskip\bigskip\bigskip

\begin{center}
{\Large\bf
Weighted Energy Decay for \bigskip\\
3D Klein-Gordon Equation}
\vspace{1cm}
\\
{\large A.~I.~Komech}
\footnote{
Supported partly
 by the Alexander von Humboldt
Research Award.
}$^{,\5 2}$
\\
{\it Fakult\"at f\"ur Mathematik, Universit\"at Wien\\
and Institute for Information Transmission Problems RAS}\\
 e-mail:~alexander.komech@univie.ac.at
\medskip\\

{\large E.~A.~Kopylova}
\footnote{Supported partly by  the grants of FWF,
DFG and RFBR.}
\\
{\it Institute for Information Transmission Problems RAS\\
B.Karetnyi 19, Moscow 101447,GSP-4, Russia}\\
e-mail:~elena.kopylova@univie.ac.at
\end{center}

\date{}

\vspace{0.5cm}
\begin{abstract}
\noindent
We obtain a dispersive long-time decay in weighted energy
norms for solutions of the 3D   Klein-Gordon equation with generic
potential.
The decay extends the results obtained by Jensen and
Kato for the 3D Schr\"odinger equation.
For the proof we modify the spectral approach of  Jensen and
Kato to make it applicable to relativistic equations.

\smallskip

\smallskip

\end{abstract}

\end{titlepage}

\setcounter{equation}{0}
\section{Introduction}
In this paper, we establish a dispersive long time decay
for the solutions to 3D Klein-Gordon equation
\be\la{KGE0}
  \ddot\psi(x,t)=\Delta\psi(x,t)-m^2\psi(x,t)+V(x)\psi(x,t),
\quad x\in\R^3,
  \quad m> 0
\ee
in weighted energy norms. In vectorial form,
equation (\ref{KGE0}) reads
\be\la{KGEr} i\dot \Psi(t)=\cH \Psi(t)
\ee
where
\be\la{H}
 \Psi(t)=\left(\begin{array}{c}
  \psi(t)
  \\
  \dot\psi(t)
  \end{array}\right),
~~~~~~~
\cH =\left(\begin{array}{cc}
  0               &   i
  \\
  i(\Delta-m^2+V)   &   0
  \end{array}\right)
\ee
For $s,\si\in\R$,
let us denote by $H^s_\si=H^s_\si (\R^3)$
the weighted Sobolev spaces introduced by Agmon, \ci{A},
with the finite norms
$$
  \Vert\psi\Vert_{H^s_\si}=\Vert\langle x
\rangle^\si\langle\na\rangle^s\psi\Vert_{L^2}<\infty,
\quad\quad \langle x\rangle=(1+|x|^2)^{1/2}
$$
We assume that $V(x)$ is a real function, and
\be \label{V}
    |V(x)|+|\na V(x)|\le C\langle x\rangle^{-\beta},~~~~~x\in\R^3
\ee
for some $\beta>3$. Then the multiplication by $V(x)$
is bounded operator $H^1_s \to H^{1}_{s+\beta}$ for any $s\in\R$.

We restrict ourselves to the ``regular case'' in the terminology of \cite{jeka} 
(or ``nonsingular case'' in  \cite{M}) which holds for {\it generic potentials}.
Equivalently, the truncated resolvent of the Schr\"odinger operator $H=-\De+V(x)$ 
is bounded at the end point $\lam=0$  of the continuous spectrum 
by \cite[Theorem 7.2]{M}. In other words, 
the point $\lam=0$ is neither eigenvalue nor resonance for the operator $H$.

\begin{defin}\la{def2}
  $\cF _{\si}$ is the  complex  Hilbert space $H^1_{\si}\oplus H^0_{\si}$
of vector-functions $\Psi =(\psi ,\pi )$
with the norm
  \be\la{Falfa}
   \Vert \,\Psi\Vert_{\cF _{\si}}=
   \Vert\psi\Vert_{H^1_\si} +\Vert\pi \Vert_{H^0_\si}<\infty
  \ee
\end{defin}

Our main result is the following long time decay of the solutions
to (\re{KGEr}): in the ``regular case'',
\be\label{full}
 \Vert\cP_c\Psi(t)\Vert_
 {\cF_{-\si}}=\cO (|t|^{-3/2}),\quad t\to\pm\infty
\ee
for initial data $\Psi_0=\Psi(0)\in\cF _\si$ with $\sigma>5/2$
where  $\cP_c$ is a Riesz projector onto the continuous spectrum
of the operator $\cH$.
The decay is desirable for the study of asymptotic stability and scattering
for the solutions to nonlinear hyperbolic equations. The study has been started in 90'
for nonlinear Schr\"odinger equation, \ci{BP1995,PW92,PW94,SW90,SW92},
and continued last decade \ci{BS2003,Cu2001,KKop2006}.
The study has been extended to the Klein-Gordon equation in \ci{IKV,SW99}.
Further extension need more information on the decay for the corresponding linearized
equations that stipulated our investigation.


Let us comment on previous results in this direction. Local energy decay
has been established first in  the scattering theory for linear Schr\"odinger equation 
developed since 50' by Birman, Kato, Simon, and others.

For free 3D Klein-Gordon equation, the decay $\sim t^{-3/2}$ in $L^\infty$ norm
has been proved first by Morawetz and Strauss \ci[Appendix B]{MS72}.
For wave and Klein-Gordon equations with magnetic potential,
the decay $\sim t^{-3/2}$  has been established primarily by Vainberg \ci{V1976} in local 
energy norms for initial data with compact support.
The results were extended to general hyperbolic partial differential
equations by Vainberg in \ci{V1989}.
The decay in the $L^p$ norms for wave and Klein-Gordon equations has been obtained in
\ci{Brenner79, Brenner85, Delort2001,Klainerman93, MSW80, Wed03, Y95}.

However, applications to asymptotic stability of  solutions to the nonlinear equations 
also require an exact characterization of the decay for the corresponding
linearized equations in weighted norms (see e.g.\ci{BP1995,BS2003,Cu2001, SW99}).

The decay of type (\re{full}) in weighted norms has been established first
by Jensen and Kato \ci{jeka} for the Schr\"odinger equation in the dimension $n=3$. 
The result has been extended to all other dimensions by Jensen and
Nenciu \ci{je1,je2,JN2001},and to more general PDEs of the Schr\"odinger type 
by Murata \ci{M}. The survey of the results can be found in \ci{Schlag}.

For free wave equations corresponding to $m=0$, some estimates in weighted $L^p$-norms 
have been established in \ci{AGK}. The Strichartz weighted estimates for
the perturbed Klein-Gordon equations were established in \ci{KL2007}.

For the free 3D Klein-Gordon equation, the decay  (\re{full}) in the weighted energy 
norms has been proved first in \ci[Lemma 18.2]{IKV}. However, for the perturbed 
relativistic equations the decay was not proved until now.
The problem was that the Jensen-Kato approach is not applicable directly
to the relativistic equations. The difference reflects distinct character
of wave propagation in the relativistic and nonrelativistic equations (see below).

Let us comment on the disctinction and our techniques. The Jensen-Kato approach
\ci{jeka} relies on the spectral Fourier-Laplace representation
\be\la{FL}
P_c\Psi(t)=\fr 1{2\pi i}\int\limits_0^\infty
e^{-i\om t}\Big[R (\om+i0)-R (\om-i0)\Big] \Psi_0d\om,~~~~~~t\in\R
\ee
where $R (\om)$ is the resolvent  of the Schr\"odinger operator $H=-\Delta+V$, 
and $P_c$ is the corresponding projector onto the continuous spectrum of $H$.
Integration by parts implies the time decay of type (\re{full})since the resolvent 
$R (\om)$ is sufficiently smooth and its derivatives $\pa^k_\om R (\om)$ have a good 
decay at $|\om|\to\infty$ for large $k$ in the weighted norms. On the other hand,
in the case of the Klein-Gordon, the derivatives do not decay though the smoothness 
of the resolvent also follows from the results \ci{jeka}.

Let us illustrate this difference in the case of the corresponding
free 3D equations:\\
i) the resolvent of the free Schr\"odinger equation is the integral operator with 
the kernel
$$
R_{\rm S}(\om,x-y)=\fr{e^{i\sqrt{\om}|x-y|}}{4\pi|x-y|}
$$
ii) the resolvent of the free Klein-Gordon equation
is the integral operator with the matrix kernel
\be\la{fKGi}
R_{\rm KG}(\om,x-y)=
\left(\begin{array}{cc}
  0         &   0
  \\
  -i\de(x-y) &  0
  \end{array}\right)
+
\fr {e^{i\sqrt{\om^2-m^2}|x-y|}} {4\pi|x-y|}
\left(\begin{array}{cc}
  \om         &   i
  \\
  -i\om^2 &  \om
  \end{array}\right)
\ee
and the region of integration in the corresponding formula (\re{FL}) is changed to 
$|\om|>m$. Leading singularities of the both resolvents are almost identical:
$\sqrt{\om}$ at $\om=0$ for $R_{\rm S}$, and $\sqrt{\om\mp m}$ at $\om=\pm m$ for 
$R_{\rm KG}$. Hence, the contribution of law frequencies into the integral (\re{FL}) 
decays like $t^{-3/2}$ both for the Schr\"odinger and Klein-Gordon case.

Now let us discuss the contribution  of high frequencies into the integral (\re{FL}). 
For the Schr\"odinger case, the contribution decays like $\sim t^{-N}$ with any $N>0$. 
This follows by partial integration since the derivatives
$\pa_\om^k R_{\rm S}(\om,x-y)$ decay like $|\om|^{-k/2}$ as $\om\to \infty$.

On the other hand, the kernel  $R_{\rm KG}(\om,x-y)$ does not decay for large $|\om|$,
and differentiation in $\om$ does not improve the decay (cf. the bounds (\re{bR0}) 
and (\re{bR})). Hence,  for the Klein-Gordon equation the integration by
parts does not provide the long time decay.

This difference is not only technical. It reflects the fact that the multiplication 
by $t^N$, with large $N$, improves the smoothness of the solutions to the Schr\"odinger 
equation in contrast to the Klein-Gordon equation. This corresponds to distinct
character of the wave propagation in the relativistic and nonrelativistic equations:
\\
i) for a  solution $\psi(x,t)$ to the Schr\"odinger equation,  main singularity is 
concentrated at $t=0$ and disappears at infinity for $t\ne 0$ due to
infinite speed of propagation.
\\
ii) for a solution  $\psi(x,t)$ to the Klein-Gordon equation, the singularities move
with bounded speed, thus they are present forever in the space.

Thus, the proof of the decay for the high energy component of the solution
requires novel robust ideas. This problem is resolved at present paper with
a modification of the Jensen and Kato technique. Our modification relies 
on a version of the  Huygens principle, the  Born series and the  convolution.
Namely, the resolvent $\cR(\om)$ of the operator $\cH$ admits the finite Born expansion
\be\la{idi}
  \cR (\om)
  = \cR _0(\omega)-\cR _0(\omega)\cV \cR _0(\omega)
  +\cR _0(\omega)\cV \cR _0(\omega)\cV \cR (\omega)
\ee
where $\cR_0(\om)$ stands for the free resolvent with the integral kernel (\re{fKGi})
corresponding to $V=0$, and $\cV=\left(
\ba{cc}0&0\\
V&0
\ea
\right)$.
Taking the inverse Fourier-Laplace transform, we obtain the corresponding expansion 
for the dynamical group $\cU(t)$ of the Klein-Gordon equation
(\re{KGEr}),
\be\la{idid}
 \cU(t)= \cU_0(t)+i\int_0^t \cU_0(t-s)\cV \cU_0 (s)ds
  -iF^{-1}_{\om\to t}\Big[
\cR _0(\omega)\cV \cR _0(\omega)\cV \cR (\omega)\Big]
\ee
where  $\cU_0(t)$ stands for the free dynamical group corresponding to $V=0$. 
The expansion corresponds to iterative procedure in solving  the perturbed Klein-Gordon 
equation (\re{KGEr}).
Further we consider separately each term in the right hand side of (\re{idid}):

I. As we noted above, for the first term $\cU_0(t)$ we cannot deduce the time 
decay (\re{full}) from the spectral representation of type (\re{FL}). On the other hand, 
the decay has been established in \ci[Lemma 18.2]{IKV} using an analog of the strong 
Huygens principle extending Vainberg's trick \ci{V1989} from the wave
to the Klein-Gordon equation.

II. For the second term we also  cannot deduce the time decay from the spectral
representation. However, the  decay follows by standard estimates for the convolution 
using the decay of the first term and the condition (\ref{V}) on the  potential.

III. Finally, the time decay for the last term  follows from  the spectral
representation by the Jensen-Kato technique  since
$\Vert\cV\cR _0(\omega)\cV\Vert\sim|\om|^{-2}$ as $|\om|\to\infty$
that follows from the (expected) lucky structure of the matrix
$\cV\cR _0(\omega)\cV$ (see (\re{ident})).


Our paper is organized as follows.
In Section \ref{fKG} we obtain the time decay for the solution to the free Klein-Gordon 
equation and state the spectral properties of the free resolvent which follow from the 
corresponding known properties of the free Schr\"odinger resolvent.
In Section \ref{pKG} we obtain spectral properties of the perturbed resolvent
and prove the decay \eqref{full}.
In Section 4 we apply the obtained decay to the asymptotic completeness.

In Appendix A we prove a revised version of Agmon-Jensen-Kato high energy decay
for  the free Schr\"odinger resolvent which we use in Section 2.
Finally, in Appendix B we give a streamlined proof of the Jensen-Kato lemma
on the decay of the Fourier integrals which we need in Section 3.

The asymptotic decay (\re{full}) is proved in \cite{KKop2009} for 1D  Klein-Gordon
equation. For the 3D wave equation corresponding to $m=0$, the weighted energy 
decay of type (\re{full}) was established in \ci{Kop09}.

\setcounter{equation}{0}
\section{Free Klein-Gordon equation}\la{fKG}
\subsection{Time decay}
First, we prove the time decay
 (\ref{full}) for
the free Klein-Gordon equation:
\be\la{KGE}
  \ddot\psi(x,t)=\Delta\psi(x,t)-m^2\psi(x,t),\quad x\in\R^3,\quad t\in\R
\ee
In vectorial form  equation (\ref{KGE}) reads
\be\la{KGE1}
i\dot \Psi(t)=\cH _0\Psi(t)
\ee
where
\be\la{H_0}
\Psi(t)=\left(\begin{array}{c}
  \psi(t)
  \\
  \dot\psi(t)
  \end{array}\right),
~~~~~~~~
\cH _0=\left(\begin{array}{cc}
  0               &   i
  \\
  i(\Delta-m^2)   &   0
  \end{array}\right)
\ee
Denote by $\cU _{0}(t)$ the dynamical  group of the equation
(\ref{KGE1}).
It is strongly continuous group
in the Hilbert space $\cF_0$.
The group is unitary after a suitable modification of the norm
that follows from the energy conservation.

\begin{pro} \la{TD} (cf.\cite{IKV}, Lemma 18.2)
  Let  $\si>3/2$. Then for $\Psi_0\in\cF_\si$
  \be\la{lins}
  \Vert\cU _0(t)\Psi_0\Vert_{\cF _{-\si}}\le\fr{C\Vert \Psi_0\Vert_{\cF _{\si}}}
  {(1+|t|)^{3/2}},\quad t\in\R
  \ee
\end{pro}
\begin{proof}
{\it Step i)}
It suffices to consider $t > 0$. In this case
the matrix kernel
of the dynamical  group $\cU _{0}(t)$
can be written as
$\cU _{0}(x-y,t)$ where
\begin{equation} \label{KGsol}
  \cU _0(z ,t)=\left( \ba{ll}
  \dot U(z,t)    &              U(z,t)\\
  \ddot U(z,t)   &        \dot U(z,t)
  \ea \right),\quad z\in\R^3
\end{equation}
and
\be\la{Gbess}
 U(z,t)=\frac{\delta (t-|z|)}{4\pi t}-
\fr{m}{4\pi}\fr{\theta(t-|z|)J_{1}
 (m\sqrt{t^{2}-|z|^{2}})}{\sqrt{t^{2}-|z|^{2}}},~~~~t>  0
\ee
where $J_{1}$ is the Bessel function of order 1,
and $\theta$ is the Heavyside function.
Let us fix an arbitrary $\ve\in(0,1)$.
Well known asymptotics of the Bessel function imply that
\be\la{W}
  |\pa_{z}^{\al}\cU _{0}(z,t)| \le C(\ve)(1+t)^{-3/2},\quad
|z|\le\ve t,~~t\ge 1
\ee
for $|\al|\le 1$.\\
\textit{Step ii)}
Now we consider
an arbitrary $t\ge 1$. Let us
split the initial function $\Psi_{0}$ in two terms,
$\Psi_{0}=\Psi_{0,t}^{\prime }+\Psi_{0,t}^{\prime \prime }$ such that
\begin{equation} \label{FFn}
  \Vert \Psi_{0,t}^{\prime}\Vert _{\cF _\si}+\Vert \Psi_{0,t}^{\prime\prime }
  \Vert _{\cF _\si}\le C\Vert \Psi_{0}\Vert _{\cF _\si},\quad t\ge 1
\end{equation}
and
\be\label{F}
  \Psi_{0,t}^{\prime }(x)=0~~\mbox{for}~|x|>\frac{\ve t}{2},
  ~~~~~~~\mbox{and}~~~~~~~~~~
  \Psi_{0,t}^{\prime \prime }(x)=0~~\mbox{for}~|x|<\frac{\ve t}{4}
  \ee
The estimate (\re{lins})
for $\cU _0(t)\Psi_{0,t}^{\prime\prime}$ follows by
energy conservation for the Klein-Gordon equation,
(\ref {F})  and (\ref{FFn}):
\begin{eqnarray} \label{zxc}
  \Vert \cU _0(t)\Psi_{0,t}^{\prime\prime }\Vert _{\cF _{-\si}} &\le &
  \Vert \cU _0(t)\Psi_{0,t}^{\prime\prime }\Vert _{\cF_0}
  \leq C\Vert \Psi_{0,t}^{\prime \prime }\Vert _{\cF_0}
  \nonumber \\
\nonumber \\
  &\leq &\fr
{C_{1}(\ve)\Vert \Psi_{0,t}^{\prime \prime }\Vert _{\cF_\si}}
{(1+t)^{\si}}
  \leq \fr
{C_{2}(\ve)\Vert \Psi_{0}\Vert _{\cF _\si}}
{(1+t)^{3/2}},\quad t\geq 1
\end{eqnarray}
since $\si>3/2$.
\\
\textit{Step iii)}
Next we consider $\cU _0(t)\Psi_{0,t}^{\prime }$.
Now we split the operator $\cU _0(t)$ in two terms:
$$
  \cU _0(t)=(1-\zeta )\cU _0(t)+\zeta \cU _0(t),~~~~~t\ge 1
$$
where $\zeta $ is the operator of multiplication by the function $\zeta ({|x|}/{t})$
such that $\zeta =\zeta (s)\in C_{0}^{\infty }(\R)$,
$\zeta (s)=1$ for $|s|<\ve/4$, $\zeta (s)=0$ for $ |s|>\ve/2$.
Obviously,
for any $\al$ we have
$$
 |\partial _{x}^{\alpha }\zeta ({|x|}/{t})|\leq C<\infty,~~~~~~t\ge 1
$$
Furthermore, $1-\zeta ({|x|}/{t})=0$ for $|x|<\ve t/4$, hence
\be \label{zaz}
  \Vert(1-\zeta )\cU _0(t)\Psi_{0,t}^{\prime }\Vert_{\cF _{-\si}}
  \le
\fr{
C_{3}(\ve)\Vert(1-\zeta)\cU
_0(t)\Psi_{0,t}^{\prime}\Vert_{\cF_0}
}{(1+t)^{\si}}
\le
\fr
{C_{4}(\ve)\Vert\cU _0(t)\Psi_{0,t}^{\prime }\Vert_{\cF_0}
}{(1+t)^{\si}}
\ee
Applying here the energy conservation for the group
$\cU_0(t)$, we obtain by (\ref{FFn}) that
\be \label{zazo}
  \Vert(1-\zeta )\cU _0(t)\Psi_{0,t}^{\prime }\Vert_{\cF _{-\si}}
\le
\fr{
C_{5}(\ve)\Vert\Psi'_{0,t}\Vert_{\cF_0}
}{(1+t)^{\si}}
\le
\fr{
C_{6}(\ve)\Vert\Psi'_{0,t}\Vert_{\cF_\si}
}{(1+t)^{\si}}
\le
\fr{
C_{7}(\ve)\Vert\Psi_{0}\Vert_{\cF_\si}
}{(1+t)^{3/2}} ,~~~~~t\ge 1
\ee
since $\si>3/2$.
\\
\textit{Step iv)}
It remains to
estimate $\zeta \cU _0(t)\Psi_{0,t}^{\prime }$.
Let $\chi _{\ve t/2}$ be the characteristic function of the ball $ |x|\le \ve t/2$.
We will use the same notation for the operator of multiplication by this characteristic
function. By (\ref{F}), we have
\be\la{zeze}
  \zeta \cU _0(t)\Psi_{0,t}^{\prime}=
\zeta \cU _0(t)\chi_{\ve t/2}\Psi_{0,t}^{\prime}
\ee
The matrix kernel of the operator
$\zeta \cU _0(t)\chi _{\ve t/2}$ is equal to
$$
  \cU' _0(x-y,t)=\zeta ({|x|}/{t})\cU _0(x-y,t)\chi _{\ve t/2}(y)
$$
Since $\zeta ({|x|}/{t})=0$ for $|x|>\ve t/2 $ and $\chi _{\ve t/2}(y)=0$
for $|y|>\ve t/2,$ the estimate (\ref{W})\ implies that
\begin{equation}\label{qaz}
  |\pa _{x}^{\al}\cU' _0(x-y,t)|\le C(1+t)^{-3/2},
  \quad |\al|\le 1,\quad t\ge 1
\end{equation}
The norm of the operator
$\zeta \cU _0(t)\chi_{\ve t/2}: \cF _{\si}\rightarrow \cF _{-\si}$
is equivalent to the norm of the operator
$$
  \langle x\rangle^{-\si}\zeta \cU _0(t)\chi _{\ve t/2}(y)
\langle y\rangle^{-\si}:
  \cF_0 \rightarrow \cF_0
$$
The norm of the latter operator does not exceed the sum in $\al$, $|\al|\leq 1$,
of the norms of operators
\begin{equation}\label{1234}
  \pa_{x}^{\al}
[\langle x\rangle^{-\si}\zeta \cU _0(t)\chi_{\ve t/2}(y)
\langle y\rangle^{-\si}]:
  L^2(\R^3)\oplus L^2(\R^3)
\to
  L^2(\R^3)\oplus L^{2}(\R^3)
\end{equation}
The estimates (\ref{qaz}) imply
that operators (\ref{1234}) are Hilbert-Schmidt operators
since $\si >3/2,$ and their Hilbert-Schmidt norms do not
exceed $C(1+t)^{-3/2}$. Hence, (\ref{zeze}) and (\ref{FFn})
imply that
\begin{equation} \label{HS}
  \Vert\zeta \cU _0(t)\Psi_{0,t}^{\prime }
\Vert_{\cF _{-\si}}\le C(1+t)^{-3/2}|
  |\Psi_{0,t}^{\prime }\Vert_{\cF_\si}\le C(1+t)^{-3/2}\Vert\Psi_{0}
\Vert_{\cF _\si},\quad t\ge 1
\end{equation}
Finally, the estimates (\ref {HS}), (\ref{zaz}) and (\ref{zxc}) imply (\ref{lins}).
\end{proof}
\subsection{Spectral properties}
We state spectral properties of the free Klein-Gordon dynamical group
$\cU_0(t)$ applying known results of \ci{A,jeka} which concern the
corresponding spectral properties of the free  Schr\"odinger
dynamical group.
For $t>0$ and $\Psi_0=\Psi(0)\in\cF_0$,
the solution $\Psi(t)$ to the free Klein-Gordon equation (\re{KGE1})
admits the spectral Fourier-Laplace representation
\be\la{Gint}
  \theta(t)\Psi(t)=\fr 1{2\pi i}\int\limits_{\R}e^{-i(\om+i\ve) t}
  \cR _0(\om+i\ve)\Psi_0~d\om,~~~~t\in\R
\ee
with any $\ve>0$
where $\theta(t)$
is the Heavyside function,
$\cR _0(\omega)=(\cH _0-\omega)^{-1}$ for
$\om\in\C^+:=\{\om\in\C:\rIm~\om>0\}$ is
the resolvent of the operator $\cH _0$.
The representation follows from the stationary equation
$\om\ti\Psi^+(\om)=\cH_0\ti\Psi^+(\om)+i\Psi_0$
for the Fourier-Laplace transform
$\ti\Psi^+(\om):=\ds\int_\R \theta(t)e^{i\om t}\Psi(t)dt$, $\om\in\C^+$.
The solution  $\Psi(t)$ is  continuous bounded  function of
$t\in\R$ with the values
in $\cF_0$
by the energy conservation for the free Klein-Gordon equation (\re{KGE1}).
Hence,
 $\ti\Psi^+(\om)=-i\cR(\om)\Psi_0$
is analytic function of $\om\in\C^+$
with the values in $\cF_0$, and
 bounded
for $\om\in\R+i\ve$. Therefore, the integral (\re{Gint})
converges in the sense of distributions
 of
$t\in\R$ with the values
in $\cF_0$.
Similarly to (\re{Gint}),
\be\la{Gints}
  \theta(-t)\Psi(t)=-\fr 1{2\pi i}\int\limits_{\R}e^{-i(\om-i\ve) t}
  \cR _0(\om-i\ve)\Psi_0~d\om,~~~~t\in\R
\ee
The resolvent $\cR _0(\omega)$ can be expressed in terms of the resolvent
$R_0(\zeta)=(-\Delta-\zeta)^{-1}$ of the free Schr\"odinger operator
\begin{equation}\label{cR0}
   \cR _0(\omega)=
   \left(\begin{array}{cc}
            \omega R_0(\omega^2-m^2)           &   iR_0(\omega^2-m^2)
   \\
         -i(1 +\omega^2 R_0(\omega^2-m^2))      &   \omega R_0(\omega^2-m^2)
     \end{array}
   \right)
\end{equation}
The  free Schr\"odinger resolvent $R_0(\zeta)$ is an
integral operator with the integral kernel
\be\la{ef}
R_0(\zeta,x-y)=\exp(i\zeta^{1/2}|x-y|)/4\pi|x-y|,
\quad\zeta\in\C^+,\quad\rIm\zeta^{1/2}>0
\ee
\begin{defin}
Denote by $\cL (B_1,B_2)$ the Banach space of bounded linear operators
from a Banach space $B_1$ to a Banach space $B_2$.
\end{defin}

The explicit formula (\re{ef}) implies the properties of $R_0(\zeta)$
which are obtained in \ci[Lemmas 2.1 and 2.2]{jeka}:
\medskip\\
i) $R_0(\zeta)$ is  analytic function of $\zeta\in\C\setminus [0,\infty)$
with the values in $\cL(H^{-1}_0,H^1_0)$;
\\
ii) For $\zeta>0$, the convergence holds
$R_0(\zeta\pm i\ve)\to R_0(\zeta\pm i0)$ as $\ve\to 0+$
in $\cL(H^{-1}_\si,H^1_{-\si})$ with $\si>1/2$;
\\
iii) The asymptotics  hold for $\zeta\in\C\setminus [0,\infty)$,
\beqn\la{exp0}
\Vert R_0(\zeta)\Vert_{\cL(H^{-1}_\si,H^1_{-\si})}
&=&\cO(1),~~~~~~~~~~\zeta\to 0,~~~~~~\si>1\\
\la{dif0}
\Vert R_0^{(k)}(\zeta)\Vert_{\cL(H^{-1}_\si,H^1_{-\si})}
&=&\cO(\zeta^{1/2-k}),~~~\zeta\to 0,~~~~~~\si>1/2+k,~~~k=1,~2,...
\eeqn
\medskip\medskip\\
Let us denote $\Ga:=(-\infty,-m)\cup(m,\infty)$
Then the properties
i) -- iv) and (\re{cR0}) imply the following lemma.
\begin{lemma}\la{sp}
i) The resolvent $\cR _0(\om)$ is  analytic function of
$\om\in\C\setminus\ov\Ga$ with the values in  $\cL (\cF _0,\cF _0)$;
\\
ii) For $\om\in\Ga$, the convergence holds
$\cR_0(\om\pm i\ve)\to \cR_0(\om\pm i0)$
as $\ve\to 0+$ in  $\cL (\cF _{\si},\cF _{-\si})$ with $\si>1/2$;
\\
iii)
The asymptotics of type (\re{exp0}), (\re{dif0}) hold for
$\om\in\C\setminus\ov\Ga$,
\beqn\la{expR0}
\!\!\!\Vert \cR_0(\om)\Vert_{\cL (\cF_{\si},\cF_{-\si})}
\!\!&=&\!\!\cO(1),\quad \om\pm m\to 0,\quad \si>1\\
\la{R0dif}
\!\!\!\Vert \cR_0^{(k)}(\om)\Vert_{\cL (\cF_{\si},\cF_{-\si})}
\!\!&=&\!\!\cO(|\om\pm m|^{1/2-k}),\quad\om\pm m\to 0,\quad \si>1/2+k,
\quad k=1,~2,... 
\eeqn
\end{lemma}
Finally, we state the asymptotics of $\cR _0(\om)$ for large $\om$
which follow from the corresponding
asymptotics of $R_0$, given in Proposition \ref{AJK}.
\begin{lemma}\la{large}
The bounds hold
\begin{equation}\label{bR0}
\Vert \cR _0^{(k)}(\om) \Vert_{\cL (\cF _{\sigma},\cF _{-\sigma})}
=\cO(1),\quad |\om|\to\infty,\quad\om\in\C\setminus\Ga
\end{equation}
with  $\si>1/2+k$ for $k=0,1,2,...$.
\end{lemma}
\begin{proof}
The bounds follow from representation (\ref{cR0}) for $\cR _0(\om)$ 
and asymptotics (\ref{A}) for $R_0(\zeta)$ with $\zeta=\om^2-m^2$.
\end{proof}
\begin{cor}
\la{irep}
For $t\in\R$ and $\Psi_0\in\cF_\si$ with $\si>1$, the group
$\cU _0(t)$ admits the integral representation
\be\la{Gint1}
  \cU _0(t)\Psi_0=\frac 1{2\pi i}\int\limits_\Gamma e^{-i\om t}
  \Big[\cR _0(\om+i0)-\cR _0(\om-i0)\Big]\Psi_0~ d\om
\ee
where the integral converges in the sense of distributions of $t\in\R$
with the values in $\cF_{-\si}$.
\end{cor}
\begin{proof}
Summing up the representations  (\ref{Gint}) and  (\ref{Gints}), 
and sending $\ve\to 0+$, we obtain  (\ref{Gint1}) by the Cauchy theorem
and Lemmas \re{sp} and \re{large}.
\end{proof}
\br\la{rev}
The estimates (\re{bR0}) do not allow  obtain the decay
(\re{lins}) by partial integration in (\re{Gint1}).
This is why we deduce the decay in Section 2.1 from explicit formulas
(\re{KGsol}) and (\re{Gbess}).
\er

\setcounter{equation}{0}
\section{Perturbed Klein-Gordon equation}\la{pKG}
To prove the long time decay for the perturbed Klein-Gordon equation,
we first establish the spectral properties of the generator.
\subsection{Spectral properties}
According \ci[p. 589]{jeka}
 and \ci[formula (3.1)]{M},
let us introduce a generalized eigenspace $\bf M$
for  the perturbed Schr\"odinger operator $H=-\Delta+V$:
$$
{\bf M}=\{\psi\in H^1_{-1/2-0}:\,~(1+A_0V)\psi=0\}
$$
where $A_0$ is the operator with the integral kernel $1/4\pi|x-y|$.
Below we assume that 
\be\la{SC}
 {\bf M}=0
\ee
In \ci[p. 591]{jeka} the point $\lam=0$  is called then ``regular point'' 
for the Schr\"odinger operator $H$ (it corresponds  to the ``nonsingular case''
in \ci[Section 7]{M}). The condition holds for {\it generic potentials} $V$ 
satisfying  (\ref{V}) (see\cite[p. 589]{jeka}).

Denote by $R (\zeta)=(H -\zeta)^{-1}$, $\zeta\in\C\setminus \R$,
the resolvent of the  Schr\"odinger operator $H$.

\br
{\rm
i)
By \ci[Theorem 7.2]{M}, the condition (\ref{SC}) is equivalent
to the boundedness of the resolvent $R (\zeta)$ at $\zeta=0$ in the norm of
$\cL(H^{-1}_\si,H^1_{-\si})$ with a suitable $\si>0$.
\\
ii)
By Lemma 3.2 in \ci{jeka}, the  condition (\ref{SC}) is equivalent to absence 
of nonzero solutions $\psi\in H^1_{-\si}$, with $\si\le 3/2$,
to the equation $H\psi=0$.
\\
iii)
$N(H)\subset\bf M$ where $N(H)$ is the zero eigenspace of the operator $H$.
The imbedding is obtained in \ci[Theorem 3.6]{jeka}.
The functions from ${\bf M}\setminus N(H)$ are called
{\it zero resonance functions}.
Hence, the  condition (\ref{SC}) means that
$\lam=0$  is neither eigenvalue nor resonance for the operator $H$.
}
\er
Let us collect the properties of $R(\zeta)$ obtained in  \cite{A,jeka,M}
under conditions  (\ref{V}) and (\ref{SC}):
\medskip\\
{\bf R1.}
$R(\zeta)$ is  meromorphic function of $\zeta\in\C\setminus[0,\infty)$
with the values in $\cL(H^{-1}_0,H^1_0)$;
the poles of $R(\zeta)$ are located at a finite set of
eigenvalues $\zeta_j<0$, $j=1,...,N$, of the  operator $H$ with the
corresponding eigenfunctions $\psi_j^1(x),...,\psi_j^{\ka_j}(x)
\in H^2_s$ with any $s\in\R$, where $\ka_j$ is the multiplicity of $\zeta_j$.
\\
{\bf R2.}
For $\zeta>0$, the convergence holds
$R(\zeta\pm i\ve)\to R(\zeta\pm i0)$ as $\ve\to 0+$ in
$\cL(H^{-1}_\si,H^1_{-\si})$ with $\si>1/2$.
\\
{\bf R3.}
The asymptotics hold for $\zeta\in\C\setminus [0,\infty)$,
\beqn\la{exp}
\Vert R(\zeta)\Vert_{\cL(H^{-1}_\si, H^1_{-\si})}
&=&\cO(1),~~~~~~~\5~~\zeta\to 0,~~~~~~\si>1\\
\la{dif01}
\Vert R^{(k)}(\zeta)\Vert_{\cL(H^{-1}_\si,H^1_{-\si})}
&=&\cO(|\zeta|^{1/2-k}),~~~\zeta\to 0,~~~~~~\si>1/2+k,~~~k=1,~2
\eeqn
\br
{\rm The asymptotics (\ref{dif01}) is deduced in \ci[Remark 6.7]{jeka}
from (\ref{exp0}), (\ref{dif0}), (\ref{exp}) and the identities}
$$
  R '=(1-R V )R'_0(1-VR ),\quad
  R ''=\Big[(1-RV )R ''_0-2R 'V R '_0\Big](1-V R )
$$
\er

Further, the resolvent $\cR (\omega)=(\cH -\omega)^{-1}$ can be expressed
similarly to (\re{cR0}):
\begin{equation}\label{cR}
   \cR (\omega)=
   \left(\begin{array}{cc}
            \omega R(\omega^2-m^2)           &   iR(\omega^2-m^2)
   \\
         -i(1 +\omega^2 R(\omega^2-m^2))      &   \omega R(\omega^2-m^2)
     \end{array}
   \right)
\end{equation}
Hence, the properties {\bf R1} -- {\bf R3} imply the corresponding
properties of $\cR(\om)$:
\begin{lemma}\la{sp1}
Let  the potential $V$ satisfy conditions (\ref{V}) and (\ref{SC}).
Then
\\
i) $\cR(\om)$ is  meromorphic function of
$\om\in\C\setminus\ov\Ga$ with the values in $\cL(\cF_0,\cF_0)$;
\\
ii) The poles of $\cR(\om)$ are located at a finite set
$$
\Si=\{\om^\pm_j=\pm\sqrt{m^2+\zeta_j},\;j=1,...,N\}
$$
of eigenvalues of the operator $\cH$ with the corresponding eigenfunctions 
$\left(\ba{c}\psi^\ka_j(x)\\
\omega^{\pm}_j\psi^\ka_j(x)\ea\right)$, $\ka=1,...,\ka_j$;
\\
iii) For $\om\in\Ga$, the convergence holds
$\cR(\om\pm i\ve)\to \cR(\om\pm i0)$ as $\ve\to 0+$ in
$\cL(\cF_\si,\cF_{-\si})$ with $\si>1/2$;
\\
iv)
The asymptotics of type (\re{expR0}), (\re{R0dif}) hold
for $\om\in\C\setminus\ov\Ga$,
\beqn
\Vert \cR(\om)\Vert_{\cL (\cF_{\si},\cF_{-\si})}
\!\!&=&\!\!\cO(1),~~~~~\5~~~~\om\pm m\to 0,~~~~~~\si>1
\la{Rexp}
\\
\Vert \cR^{(k)}(\om)\Vert_{\cL (\cF_{\si},\cF_{-\si})}
\!\!&=&\!\!\cO(|\om\pm m|^{1/2-k}),~~~\om\pm m\to 0,~~~~~~\si>1/2+k,~~~k=1,~2 \la{dR}
\eeqn
\end{lemma}
Now we obtain the asymptotics of $R(\zeta)$ and $\cR (\om)$ 
for large $\zeta$ and $\om$.
\begin{lemma}\la{RRR}
Let the potential $V$ satisfy (\ref{V}).
Then for $s=0,1$ and  $l=-1,0,1$ with $s+l\in\{0;1\}$, we have
\be\la{AR}
\Vert R^{(k)}(\zeta)\Vert_{\cL (H^s_\si,H^{s+l}_{-\si})}
=\cO(|\zeta|^{-\fr{1-l+k}2}),\quad\zeta\to\infty,\quad
\zeta\in\C\setminus[0,\infty)
\ee
with $\si>1/2+k$ for  $k=0,1,2$.
\end{lemma}
\begin{proof}
The lemma follows from Proposition \ref{AJK} in appendix A
by the arguments from the proof of Theorem 9.2 in \cite{jeka},
where the bounds are proved for  $s=0$ and $l=0,1$.
\end{proof}
Hence (\ref{cR}) implies
\bc\la{cRR}
Let the potential  $V$ satisfy (\ref{V}).  Then the following bounds hold
\be\la{bR}
 \Vert\cR^{(k)}(\om)\Vert_{\cL (\cF_\si,\cF_{-\si})}=\cO(1),
 \quad |\om|\to\infty,\quad \om\in\C\setminus\Gamma
\ee
with  $\si>1/2+k$ for  $k=0,1,2$.
\ec
Finally, let us denote by $\cal V$ the matrix
\be\la{cV}
  \cV =\left(\begin{array}{cc}
  0               &   0
  \\
  iV               &   0
\end{array}\right)
\ee
Then the vectorial  equation (\ref{KGEr}) reads
\be\la{KGErr} 
i\dot \Psi(t)=(\cH_0+\cV) \Psi(t)
\ee
where $\cH_0$ is defined in (\ref{H_0}). The resolvents $\cR(\om)$, $\cR_0(\om)$ 
are related by the Born perturbation series
\be\la{id}
  \cR (\om)= \cR _0(\omega)-\cR _0(\omega)\cV \cR _0(\omega)
  +\cR _0(\omega)\cV \cR _0(\omega)\cV \cR (\omega),
  \quad\om\in\C\setminus[\Ga\cup\Si]
\ee
which follows by iteration of
$ \cR (\om)= \cR _0(\omega)-\cR _0(\omega)\cV \cR(\omega)$.
An important role in (\re{id}) plays the product $\cW(\om):= \cV\cR _0(\om)\cV$. 
We obtain the asymptotics of $\cW(\om)$ for large $\om$.
\begin{lemma}\la{large1}
Let  the potential $V$ satisfy (\ref{V}) with $\beta>1/2+k+\de$
where $\de>0$ and
$k=0,1,2$. Then  the following asymptotics hold
\begin{equation}\label{bRV}
\Vert\cW^{(k)}(\om)\Vert_{\cL (\cF _{-\de},\cF _{\de})}
=\cO(|\om|^{-2}),\quad |\om|\to\infty,\quad \om\in\C\setminus\Ga
\end{equation}
\end{lemma}
\begin{proof}
The asymptotics  follow from  the algebraic structure of the matrix
\be\la{ident}
  \cW^{(k)}(\om)=\cV\cR _0^{(k)}(\om)\cV =\left(\begin{array}{cc}
  0                                      &   0
  \\
  -iV\pa_\om^k R_0(\om^2-m^2)V            &   0
\end{array}\right)
\ee
since (\ref{A}) with $s=1$ and $l=-1$ implies that
$$
  \Vert VR _0^{(k)}(\zeta)V f\Vert_{H^{0}_{\de}}\le 
  C\Vert R_0^{(k)}(\zeta)V f\Vert_{H^{0}_{\de-\beta}}
  =\cO(|\zeta|^{-1-\fr k2})\Vert V f\Vert_{H^{1}_{\beta-\de}}
  =\cO(|\zeta|^{-1-\fr k2})\Vert f\Vert_{H^{1}_{-\si}}
$$
since $1/2+k<\beta-\de$.
\end{proof}

\subsection{Time decay}

In this section we combine the spectral properties of the perturbed resolvent
and time decay for the unperturbed dynamics using the (finite) Born
perturbation series. Our main result is the following. \\
\begin{theorem}\la{main}
Let conditions (\ref{V}) and (\ref{SC}) hold.
Then
\begin{equation}\la{full1}
   \Vert e^{-it{\cal H}}-\sum\limits_{\om_J\in\Si}
   e^{-i\omega_J t}P_J\Vert_{\cL (\cF _\si,\cF _{-\si})}
  ={\cal O}(|t|^{-3/2}),\quad t\to\pm\infty
\end{equation}
with  $\sigma>5/2$,
where  $P_J$ are the Riesz
projectors onto the corresponding eigenspaces.
\end{theorem}
\begin{proof}
{\it Step i)}
Let us substitute the series (\ref{id}) into the spectral representation
of type (\re{Gint}) for the solution to
(\ref{KGE0}) with $\Psi(0)=\Psi_0\in\cF_\si$ where $\si>3/2$.
Then Lemma \re{sp1} and asymptotics (\re{Rexp}) and (\re{bR}) with $k=0$
imply similarly to (\re{Gint1}), that
\beqn
  \Psi(t)&-&\sum\limits_{\om_J\in\Si}
e^{-i\om_J t}P_J\Psi_0=
  \fr 1{2\pi i}\int\limits_\Gamma e^{-i\om t}
  \Big[\cR (\om+i0)-\cR (\om-i0)\Big]\Psi_0~ d\om \la{sol}
 \\
  \nonumber
  &=&\fr 1{2\pi i}\int\limits_\Gamma e^{-i\om t}
  \Big[\cR _0(\om+i0)-\cR _0(\om-i0)\Big]\Psi_0~ d\om\\
  \nonumber
  &+&\fr 1{2\pi i}\int\limits_\Gamma e^{-i\om t}
\Big[\cR _0(\om+i0)
  \cV \cR _0(\om+i0)-\cR _0(\om-i0)\cV \cR _0(\om-i0)\Big]\Psi_0~ d\om\\
  \nonumber
  &+&\frac 1{2\pi i}\int\limits_\Gamma e^{-i\om t}
  \Big[[\cR _0\cV \cR _0\cV \cR] (\om+i0)
  -[\cR _0\cV \cR _0\cV \cR] (\om-i0)\Big]\Psi_0~ d\om\\
  &=&\Psi_1(t)+\Psi_2(t)+\Psi_3(t),~~~~~~t\in\R \nonumber
\eeqn
where $P_J$ stands for the corresponding Riesz projector
$$
P_J\Psi_0:=-\fr 1{2\pi i}\int_{|\om-\om_J|=\de}\cR(\om)\Psi_0d\om
$$
with a small $\de>0$.
Further we analyze each term $\Psi_k$ separately.
\\
{\it Step ii)}
The first term
$\Psi_1(t)=\cU _0(t)\Psi_0$ by  (\ref{Gint1}). Hence,
Proposition \ref{TD} implies that
\be\la{lins1}
  \Vert \Psi_1(t)\Vert_{\cF _{-\si}}
 \le \ds\fr{C\Vert \Psi_0\Vert_{\cF _\si}}{(1+|t|)^{3/2}},
  \quad t\in\R,\quad \sigma>3/2.
\ee
{\it Step iii)}
The second term $\Psi_2(t)$ can be rewritten as a convolution.
\begin{lemma}
The convolution representation holds
\be\la{P2}
  \Psi_2(t)=
  i\int\limits_0^t \cU _0(t-\tau)\cV \Psi_1(\tau)~d\tau,~~~~t\in\R
\ee
where the integral converges in $\cF_{-\si}$ with $\si>3/2$.
\end{lemma}
\begin{proof}
The term  $\Psi_2(t)$ can be rewritten as
\be\la{P22}
 \Psi_2(t)
  =
\fr 1{2\pi i}\int\limits_\R
\Big[e^{-i\om t}\cR _0(\om+i0)
  \cV \cR _0(\om+i0)
-e^{-i\om t}\cR _0(\om-i0)\cV \cR _0(\om-i0)\Big]\Psi_0~ d\om
\ee
\\
Let us denote
$$
\cU_0^\pm(t):=\theta(\pm t)\cU_0(t),~~~
\Psi_1^\pm(t):=\theta(\pm t)\Psi_1(t),~~~~~~~~~~~~t\in\R
$$
We know that $\cR _0(\om+i0)\Psi_0=i\ti\Psi_1^+(\om)$,
hence  the first term in the right hand side of (\re{P22}) reads
\beqn\la{P22r}
 \Psi_{21}(t)
  &=&
\fr 1{2\pi }\int\limits_\R
e^{-i\om t}\cR _0(\om+i0)
  \cV \ti\Psi_1^+(\om)~ d\om
\nonumber\\
\nonumber\\
&=&\fr 1{2\pi }\int\limits_\R
e^{-i\om t}\cR _0(\om+i0)
  \cV \Big[
\int_\R e^{i\om \tau}\Psi_1^+(\tau)d\tau
\Big]d\om
\nonumber\\
\nonumber\\
&=&\fr 1{2\pi }(i\pa_t+i)^2\int\limits_\R
\fr
{e^{-i\om t}}{(\om+i)^2}
\cR _0(\om+i 0)
  \cV \Big[
\int_\R e^{i\om\tau}\Psi_1^+(\tau)d\tau
\Big]d\om
\eeqn
The last double integral  converges in $\cF_{-\si}$ with $\si>3/2$
by (\re{lins1}), Lemma \re{sp} ii), and (\re{bR0}) with $k=0$. 
Hence, we can change the order of integration by the Fubini theorem. 
Then we obtain that
\be\la{p21}
\Psi_{21}(t)=i\int_\R\cU_0^+(t-\tau)\cV\Psi_1^+(\tau)d\tau=
\left\{\ba{cl}
\ds i\int_0^t\cU_0(t-\tau)\cV\Psi_1(\tau)d\tau&,~~t>0\\
0&,~~t<0
\ea
\right.
\ee
since
$$
\cU_0^+(t-\tau)=\fr 1{2\pi i}(i\pa_t+i)^2\int\limits_{\R}
\fr
{e^{-i\om (t-\tau)}}{(\om+i)^2}
  \cR _0(\om+i0)~d\om
$$
by (\re{Gint}).
Similarly, integrating the second term in the right hand side
of (\re{P22}), we obtain
\be\la{p22}
\Psi_{22}(t)=i\int_\R\cU_0^-(t-\tau)\cV\Psi_1^-(\tau)d\tau=
\left\{\ba{cl}
0&,~~t>0\\
\ds i\int_0^t\cU_0(t-\tau)\cV\Psi_1(\tau)d\tau&,~~t<0
\ea
\right.
\ee
Now (\re{P2}) follows since $\Psi_2(t)$ is the sum of two expressions
(\re{p21}) and (\re{p22}).
\end{proof}

Further, let us consider $\si\in(3/2,\beta/2]$. 
Applying Proposition \ref{TD} to the integrand in (\re{P2}), we obtain that
$$
  \Vert \cU_0(t-\tau)\cV \Psi_1(\tau)\Vert_{\cF _{-\si}}
 \le\ds\fr{C\Vert \cV \Psi_1(\tau)\Vert_{\cF _{\si}}}{(1+|t-\tau|)^{3/2}}
 \le\ds\fr{C_1\Vert \Psi_1(\tau)\Vert_{\cF _{-\si}}}{(1+|t-\tau|)^{3/2}}
  \le\ds\fr{C_2\Vert \Psi_0\Vert_{\cF _{\si}}}{(1+|t-\tau|)^{3/2}(1+|\tau|)^{3/2}}
$$
Therefore, integrating here in $\tau$, we obtain by (\re{P2}) that
\be\la{lins2}
  \Vert \Psi_2(t)\Vert_{\cF _{-\si}}\le \ds\fr{C\Vert \Psi_0\Vert_{\cF _\si}}
  {(1+|t|)^{3/2}}, \quad t\in\R,\quad \sigma>3/2
\ee
{\it Step iv)}
Finally, let us rewrite the last term in (\re{sol}) as
\be\la{PM}
  \Psi_3(t)=\frac 1{2\pi i}\int\limits_{\Gamma}e^{-i\om t}
\cN(\om)
\Psi_0~ d\om
\ee
where $\cN(\om):=\cM(\om+i0)-\cM(\om-i0)$ for
$\om\in\Ga$, and
\be\la{cM}
  \cM(\om):=\cR _0(\om)\cV
  \cR _0(\om)\cV \cR (\om)=\cR _0(\om)\cW(\om) \cR (\om),~~~~~~~~
\om\in\C\setminus [\Ga\cup\Si]
\ee
First, we obtain the asymptotics of $\cN(\om)$ at the points $\pm m$.
\begin{lemma}\la{lcM}
  {\it i)}
  The following asymptotics hold
  \begin{equation}\label{expcM}
\Vert
\cN(\om)
\Vert_{\cL (\cF _{\sigma},\cF _{-\sigma})}
=\cO (|\om\mp m|^{1/2}),\quad
  \quad \om\to \pm m,\quad\om\in\Ga
  \end{equation}
  for $\sigma>3/2$.\\
  {\it ii)}
  The asymptotics  \eqref{expcM} can be differentiated twice:
\beqn\la{dM}
\left.\ba{rcl}
\Vert
 \cN'(\om)
\Vert
_{\cL (\cF _{\sigma},\cF _{-\sigma})}
&=&\cO (|\om\mp m|^{-1/2})\\
\\
\Vert
\cN''(\om)
\Vert
_{\cL (\cF _{\sigma},\cF _{-\sigma})}
&=&\cO (|\om\mp m|^{-3/2})
\ea\right|
 \quad \om\to \pm m,\quad\om\in\Ga
\eeqn
for $\sigma>5/2$.
\end{lemma}
\begin{proof}
The lemma follows from the corresponding
asymptotics (\ref{expR0}), (\ref{R0dif}) and (\ref{Rexp}), (\ref{dR})
of the resolvents
$\cR _0$ and $\cR $ and their derivatives, and
assumption (\ref{V}) on the potential $V(x)$.
\end{proof}
Second, we obtain the asymptotics of $\cN(\om)$ and its
derivatives for large $\om$.
 \begin{lemma}\la{bM} For $k=0,1,2$ the asymptotics  hold
 \be\la{expbM}
  \Vert\cN^{(k)}(\om)\Vert_{\cL (\cF _{\sigma},\cF _{-\sigma})}
  =\cO(|\om|^{-2}),\quad |\om|\to\infty,\quad \om\in\Ga
\ee
with  $\si>1/2+k$.
\end{lemma}
\begin{proof}
The asymptotics (\ref{expbM}) follow from the asymptotics (\ref{bR0}),
(\ref{bR}) and (\ref{bRV})
for $\cR_0^{(k)}(\om)$,  $\cR^{(k)}(\om)$ and $\cW^{(k)}(\om)$.
For example, let us consider  the case $k=2$.
Differentiating (\re{cM}), we obtain
\be\la{M2}
\cM''=\cR_0''\cW\cR+\cR_0\cW''\cR+\cR_0\cW\cR''
+2\cR_0'\cW'\cR+2\cR_0'\cW\cR'+2\cR_0\cW'\cR'
\ee
For a fixed $\si>5/2$, let us choose $\de\in (5/2,\,\min\{\si,\beta-1/2\})$.
Then for the first term in (\ref{M2}) we obtain
by (\ref{bR}) and (\ref{bRV})
\beqn\nonumber
\!\!\!&&\!\!\!\Vert\cR_0''(\om)\cW(\om)\cR(\om) f\Vert_{\cF_{-\si}}
\le
\Vert\cR_0''(\om)\cW(\om)\cR(\om) f\Vert_{\cF_{-\de}}\le
C\Vert\cW(\om)\cR(\om) f\Vert_{\cF_{\de}}
\\\\
\nonumber
\!\!\!&=&\!\!\!\cO(|\om|^{-2})\Vert\cR (\om)f\Vert_{\cF_{-\de}}
=\cO(|\om|^{-2})\Vert f\Vert_{\cF_{\de}}
=\cO(|\om|^{-2})\Vert f\Vert_{\cF_{\si}},\quad |\om|\to\infty,\quad
\om\in\C\setminus\Ga
\eeqn
Other  terms can be estimated similarly
choosing an appropriate value of $\de$. Namely,
$\de\in (1/2,\,\min\{\si,\beta-5/2\})$ for the second term,
$\de\in (5/2,\,\min\{\si,\beta-1/2\})$ for the third,
$\de\in (3/2,\,\min\{\si,\beta-3/2\})$ for the forth and sixth terms,
and $\si'\in (3/2,\,\min\{\si,\beta-1/2\})$ for the fifth term.
\end{proof}
Now we prove the desired decay of $\Psi_3(t)$ from (\re{PM})
using methods \cite{jeka}.
Let us consider the integral over $(m,\infty)$.
The integral over $(-\infty,-m)$ can be dealt in the same way.

Let us split $\Psi_3(t)$ into the low and high energy components.
We choose $\phi_1(\om),\phi_2(\om)\in C_0^{\infty}(\R)$
where $\supp \phi_1\subset [m/2,b]$ with sufficiently large $b>0$,
and  $\supp\phi_2\subset[b-1,\infty)$, such that
$\phi_1(\om)+\phi_2(\om)=1$ for $\om\in[m,\infty)$.
Then (\re{PM}) implies that $\Psi_3(t)=\Psi_{31}(t)+\Psi_{32}(t)$, where
$$
\Psi_{31}(t)=\frac 1{2\pi i}\int\limits_m^be^{-i\om t}\phi_1(\om)
\cN(\om)\Psi_0~ d\om,~~~~~
\Psi_{32}(t)=\frac 1{2\pi i}~\int\limits_{b-1}^\infty e^{-i\om t}
\phi_2(\om)\cN(\om)\Psi_0~ d\om
$$
By Lemma \ref{lcM}, we can apply to the Fourier integral $\Psi_{31}(t)$
the corresponding version of Lemma \ref{jk} below with $a=m$, 
operator function $F=\phi_1(\om)\cN(\om)$, 
and the Banach space ${\bf B}= \cL (\cF _{\si},\cF _{-\si})$
with  $\si>5/2$. Then we obtain that
\be\la{Zy}
\Vert \Psi_{31}(t)\Vert_{\cF _{-\si}}
 \le \ds\fr{C\Vert \Psi_0\Vert_{\cF _\si}}{(1+|t|)^{3/2}},
  \quad t\in\R
\ee
Further, $\supp \phi_2\cN\subset [b-1,\infty$), and
$(\phi_2\cN)''\in L^1(b-1,\infty;\cL (\cF _{\si},\cF _{-\si}))$
with $\si>5/2$  by  Lemma \ref{bM}.
Hence, two times partial integration implies that
$$
\Vert \Psi_{32}(t)\Vert_{\cF _{-\si}}
 \le \ds\fr{C\Vert \Psi_0\Vert_{\cF _\si}}{(1+|t|)^{2}},
  \quad t\in\R
$$
This completes the proof of Theorem \ref{main}.
\end{proof}
\bc
The asymptotics (\re{full1}) imply (\re{full}) with the projector
\be\la{Pr}
\cP_c:=1-\cP_d,\quad \cP_d=\sum_{\om_J\in\Si}P_J
\ee
\ec
\setcounter{equation}{0}
\section{Application to the asymptotic completeness}
We apply the obtained results to prove the asymptotic completeness
by standard Cook's argument.
\begin{theorem}\la{sc}
Let conditions  (\re{V}) and (\re{SC}) hold. Then
\\
i)  For solution to (\ref{KGEr}) with any initial function  $\Psi(0)\in\cF_0 $,
the following long time asymptotics  hold,
\be\la{scat}
  \Psi(t)=
  \sum\limits_{\om_J\in\Si}  e^{-i\om_J t}\Psi_J
  +\cU _0(t)\Phi_\pm+r_\pm(t)
\ee
where $\Psi_J$ are the corresponding eigenfunctions,
$\Phi_\pm\in\cF_0$ are the scattering states, and
\be\la{rem0}
   \Vert r_\pm(t)\Vert_{\cF_0}\to 0,~~~~~~t\to\pm\infty
\ee
ii) Furthermore,
\be\la{rem}
   \Vert r_\pm(t)\Vert_{\cF_{0}} =\cO (|t|^{-1/2})
\ee
if $\Psi(0)\in\cF_\si$ with $\si\in(5/2,\beta]$.
\end{theorem}
\begin{proof} Denote ${\cal X}_d:=\cP_d\cF_0$,  ${\cal X}_c:=\cP_c\cF_0$.
For $\Psi(0)\in{\cal X}_d$ the asymptotics  (\ref{scat}) obviously hold 
with $\Phi_\pm=0$ and $r_\pm(t)=0$. 
Hence, it remains to prove for $\Psi(0)\in {\cal X}_c$ the asymptotics 
\be\la{scat1}
\Psi(t)=\cU_0(t)\Phi_\pm+r_\pm(t)
\ee
with the remainder satisfying (\re{rem0}). Moreover, it suffices to prove 
the asymptotics (\ref{scat1}), (\re{rem}) for $\Psi_0\in{\cal X}_c\cap \cF _\si$
with $\si>5/2$ since the space $\cF _\si$ is dense in $\cF_0$, while 
the group $\cU _0(t)$ is unitary in $\cF_0 $ after a suitable modification 
of the norm. In this case Theorem \re{main} implies the decay 
\be\label{fullp}
 \Vert \Psi(t)\Vert_{\cF_{-\si}}\le C (1+|t|)^{-3/2}\Vert \Psi(0)\Vert_ {\cF_{\si}},
\quad t\to\pm\infty
\ee
We also can assume $\beta\ge \si$.

The function $\Psi(t)$ satisfies the equation (\re{KGErr}),
$$ 
i\dot \Psi(t)=(\cH_0+\cV) \Psi(t)
$$
Hence, the corresponding Duhamel equation reads
\be\la{Dug}
   \Psi(t)= \cU _0(t)\Psi(0)+
   \int\limits_0^t \cU _0(t-\tau)\cV\Psi(\tau)d\tau, ~~~~t\in\R
\ee
Let us rewrite  \eqref{Dug} as
\be\la{Dug1}
   \Psi(t)=\cU _0(t)\Big[\Psi(0)+\int\limits_0^{\pm\infty}
   \cU _0(-\tau)\cV\Psi(\tau)d\tau\Big]
   -\int\limits_t^{\pm\infty} \cU _0(t-\tau)\cV\Psi(\tau)d\tau
    =\cU _0(t)\Phi_\pm+r_\pm(t)
\ee
It remains to prove that $\Phi_\pm\in\cF_0$ and (\ref{rem}) holds.
Let us consider the sign ``+'' for the concreteness.
The  ``unitarity'' of $\cU _0(t)$   in $\cF_0 $, the condition (\re{V})
and the decay (\ref{fullp}) imply that
\beqn\la{Dug2}
   \int\limits_0^{\infty}\Vert \cU _0(-\tau)\cV
   \Psi(\tau)\Vert_{\cF_0 }d\tau &\le& C\int\limits_0^{\infty}\Vert \cV
   \Psi(\tau)\Vert_{\cF_0 }d\tau\le C_2\int\limits_0^{\infty}\Vert
   \Psi(\tau)\Vert_{\cF _{-\si}}d\tau\\
   \nonumber
   &\le& C_2\int\limits_0^{\infty}(1+\tau)^{-3/2}\Vert \Psi(0)\Vert_{\cF _{\si}}d\tau
   <\infty
\eeqn
since $|V(x)|\le C' \langle x\rangle^{-\beta}\le C'' \langle x\rangle^{-\si}$.
Hence, $\Phi_+\in \cF_0$. The estimate (\re{rem}) follows similarly.
\end{proof}
\begin{remark}
{\it i)} The asymptotic completeness is proved  
by another methods in \ci{Lundberg,Schechter,Weder}
for more general Klein-Gordon equations with an external Maxwell field.
\\
{\it ii)} A version of Theorem \ref{sc} using standard $L^p$ spaces and Strichartz 
estimates,  follows also from \cite{Y95}. Notice that the hypotheses in \cite{Y95} 
can be relaxed to (\re{V}) by the methods of \cite{GS}.
\end{remark}
\appendix

\setcounter{equation}{0}

\protect\renewcommand{\thesection}{\Alph{section}}
\protect\renewcommand{\theequation}{\thesection.\arabic{equation}}
\protect\renewcommand{\thesubsection}{\thesection.\arabic{subsection}}
\protect\renewcommand{\thetheorem}{\Alph{section}.\arabic{theorem}}
\section{Appendix: Decay of the  free Schr\"odinger resolvent}
We revise the Agmon-Jensen-Kato decay of the resolvent
\cite[(A.2')]{A}, \cite[(8.1)]{jeka}
for special case of free Schr\"odinger equation in arbitrary dimension $n\ge 1$.

\begin{pro}\la{AJK}
For  $k=0,1,2,...$ and $\si>1/2+k$ the asymptotics hold 
\be\la{A}
  \Vert R_0^{(k)}(\zeta)\Vert_{\cL (H^s_\si,H^{s+l}_{-\si})}
  = \cO(|\zeta|^{-\fr{1-l+k}2}),\quad |\zeta|\to\infty,\quad
  \zeta\in\C\setminus[0,\infty),\quad s\in\R
\ee
where $l=-1,0,1,2$ for $k=0$, and $l=-1,0,1$ for $k=1,2,..$
\end{pro}
We give a complete proof of  the asymptotics (\re{A}) refining the arguments in
the proof of Theorem A.1 from \ci[Appendix A]{A}.
Namely, we deduce Proposition \re{AJK} from the following two lemmas.
The first lemma is well known (see \ci[Lemma A.2]{A}, 
and \ci[Lemma 4, p. 442]{RS}). Denote $\pa_j=\frac{\pa}{\pa x_j}$.
\begin{lemma}\la{Agm2}
  For $\sigma>1/2$, the following inequality holds for $\psi \in
  C_0^\infty(\R^n)$
  \be\la{A2}
    \Vert\pa_j \psi \Vert_{H^0_{-\si}}
    \le C(\si)\Vert (\Delta+\zeta)\psi \Vert_{H^0_\si},
    \quad\zeta\in\C
  \ee
\end{lemma}
Second lemma is a refinement, for special case of free Schr\"odinger equation,
of Lemma A.3 from \cite[Appendix A]{A} which is proved for bounded $|\zeta|$.
\begin{lemma}\la{Agm3}
  For any $\de\in\R$ and $\psi \in C_0^\infty(\R^n)$ the  estimate holds  
  \be\la{A3}
   \Vert \psi \Vert_{H^l_\de}\le C(s)|\zeta|^{-\frac{1-l}2}
    \Big(\Vert(\Delta+\zeta)\psi \Vert_{H^0_\de}
    +\sum\limits_{j=1}^{n}\Vert\pa_j \psi (x)\Vert_{H^0_\de}\Big),
    \quad \zeta\in\C,\quad |\zeta|\ge 1, \quad l=0,1
\ee
\end{lemma}
\begin{proof}
We  will prove (\ref{A3}) for $\de=0$, and the extension to all $\de\in\R$ follows 
by the arguments from \cite[pp 207-208]{A}.
For the proof we use the bound  (cf. \ci[formula (A.15')]{A})
\be\la{A4}
  (1+|\xi|^l)^2\le C|\zeta|^{-(1-l)} \Big(
  \Big||\xi|^2-\zeta\Big|^2+|\xi|^2\Big),
  \quad \xi\in\R^n,\quad |\zeta|\ge 1,\quad l=0,1
\ee
For  $l=1$ the bound is obvious. For $l=0$ it reduces to a quadratic inequality
for $y=|\xi|^2-|\zeta|\ge -|\zeta|$ since then
$$
 \Big||\xi|^2-\zeta\Big|^2+|\xi|^2 \ge \Big||\xi|^2-|\zeta|\Big|^2+|\xi|^2
 =y^2+y+|\zeta|\ge \min\limits_{y\ge -|\zeta|}(y^2+y)
 +|\zeta|\ge \fr{|\zeta|}2,~~~|\zeta|\ge 1
$$
Finally, let us multiply both sides of (\ref{A4}) by $|\hat \psi (\xi)|^2$ 
and integrate over $\R^n$. 
Then using Parseval's formula, we find for $|\zeta|\ge 1$ that
\be\la{A5}
 \sum\limits_{|\al|\le l}\Vert D^{\al}\psi \Vert^2\le C
 \int\limits_{\R^n}(1+|\xi|^l)^2|\hat \psi (\xi)|^2d\xi\le C_1|\zeta|^{-(1-l)}
 \Big(\Vert(\Delta+\zeta)\psi \Vert^2+\sum\limits_{j=1}^{n}\Vert
 {\pa_j}\psi (x)\Vert^2\Big)
\ee
\end{proof}
~\\
{\bf Proof of Proposition \ref{AJK}}
It suffices to verify the case  $s=0$ since $R_0(\zeta)$ commutes with
the operators $\langle \na\rangle^s$ with arbitrary $s\in\R$.\\
{\it Step i)}
First,  we prove (\re{A}) with $k=0$ and  $l=0,1$
similarly to the proof of Theorem A.1 in \ci[p. 208]{A}.
Applying Lemma \ref{Agm3} with $\de=-\si$, we obtain
\be\la{A10}
 \Vert \psi \Vert_{H^l_{-\si}}
 \le C(\si)|\zeta|^{-\fr{1-l}2}\Big(\Vert (\Delta+\zeta)\psi
\Vert_{H^0_{-\si}}
 +\sum\limits_{j=1}^{n}\Vert {\pa_j}\psi \Vert_{H^0_{-\si}}\Big),\quad
 |\zeta|\ge 1,\quad l=0,1.
\ee
for all $\psi \in H^2_{\si}(\R^n)$.
On the other hand, Lemma \ref{Agm2} implies that
\be\la{A11}
  \sum\limits_{j=1}^{n}\Vert {\pa_j}\psi \Vert_{H^0_{-\si}}\le C_1(\si)
  \Vert (\Delta+\zeta)\psi \Vert_{H^0_\si},\quad j=1,...,n.
\ee
Combining (\ref{A10}) and (\ref{A11}), we obtain
$$
  \Vert \psi \Vert_{H^l_{-\si}}
  \le C(\si)|\zeta|^{-\fr{1-l}2}\Big(\Vert (\Delta+\zeta)\psi\Vert_{H^0_{-\si}}
  +C_1(\si)\Vert (\Delta+\zeta)\psi \Vert_{H^0_\si}\Big)
  \le C_2(\si)|\zeta|^{-\fr{1-l}2}\Vert (\Delta+\zeta)\psi \Vert_{H^0_\si}
$$
and then (\ref{A}) with  $k=0$ and $l=0,1$  is proved.
\\
{\it Step ii)} Second, we prove (\re{A}) in the case  $k=0$ and $l=-1$.
We use the identity $R_0(\zeta)=-(1+\Delta R_0(\zeta))/\zeta$.
The bound  with $l=1$ implies that
$\Vert R_0(\zeta)\Vert_{\cL (H^0_\si,H^1_{-\si})}=\cO (1)$,
hence $\Vert \Delta R_0(\zeta)\Vert_{\cL (H^0_\si,H^{-1}_{-\si})}=\cO (1)$.
Therefore
$$
\Vert R_0(\zeta)\Vert_{\cL (H^0_\si,H^{-1}_{-\si})}
=\Vert (1+\Delta R_0(\zeta))/\zeta\Vert_{\cL(H^0_\si,H^{-1}_{-\si})}
=\cO(|\zeta|^{-1})
$$
{\it Step iii)} Third,  we prove (\re{A}) in the case  $k=0$ and  $l=2$.
Using the identity $(1-\Delta)R_0(\zeta)=1+(1+\zeta)R_0(\zeta)$,
we obtain
\beqn\nonumber
\Vert R_0(\zeta)\Vert_{\cL (H^0_\si,H^2_{-\si})}
&=&\Vert (1-\Delta)R_0(\zeta)\Vert_{\cL (H^0_\si,H^0_{-\si})}
=\Vert 1+(1+\zeta)R_0(\zeta)\Vert_{\cL (H^0_\si,H^0_{-\si})}\\
\la{002}
&=&1+\cO(|\zeta|)\Vert R_0(\zeta)\Vert_{\cL (H^0_\si,H^0_{-\si})}
=\cO(|\zeta|^{1/2}) 
\eeqn
{\it Step iv)}
Finally, we consider the case  $k\ge 1$.
The asymptotics (\ref{A}) wits $k=1$ follows  from  the asymptotics (\ref{A}) with
$k=0$  and the Lavine-type  identity \ci[(8.2)]{jeka}
\be\la{Lav}
\zeta R'_0(\zeta)=-R_0(\zeta)+\fr 12[x\cdot\nabla,R_0(\zeta)],
\quad\zeta\in\C\setminus[0,\infty)
\ee
(where $[\cdot,\cdot]$ stands for the commutator)
since
$$
x\in \cL (H^s_\si,H^s_{\si-1}),\quad\nabla\in \cL (H^s_\si,H^{s-1}_\si)
$$
For $k\ge 2$ the asymptotics (\ref{A}) follow by induction from the recurrent relation
\ci[(8.5)]{jeka}
$$
2\zeta R_0^{(k)}(\zeta)=-(2k-3) R_0^{(k-1)}(\zeta)
-\fr 12[x,[x,R_0^{(k-2)}(\zeta)]]
$$
where the double commutator is defined as
$$
[x,[x,R]]=|x|^2R-2\sum_j x_jRx_j+R|x|^2
$$
~~~~~~~~~~~~~~~~~~~~~~~~~~~~~~~~~~~~~~~~~~~~~~~~~~~~~~~~~~~~~~~~~~~~~~~~~
~~~~~~~~~~~~~~~~~~~~~~~~~~~~~~~~~~~~~~~~~~~~~~~~~$\Box$
\setcounter{equation}{0}
\section{Appendix: the Jensen-Kato lemma}
We prove a lemma concerning the decay of the Fourier integrals which we have used 
in (\re{Zy}). The lemma is  a special case of \cite[Lemma 10.2]{jeka}, 
and our proof is a streamlined version of the proof from \ci{jeka}.
Let ${\bf B}$ denote a Banach space with the norm $\Vert\cdot\Vert$, and $b>a$.
\begin{lemma}\label{jk}
Let $F\in C(a, b; {\bf B})$ satisfy
\be\la{Zc}
  F(a)=F(b)=0,~~~~ F''\in L^1(\de, b; {\bf B}),~~~\forall\de>0,~~~~
  \Vert F''(\om)\Vert=\cO (|\om-a|^{-3/2}),~~~\om\to a
\ee
Then
\be\la{Zyg}
  \int\limits_a^b e^{-it\omega}F(\omega)d\omega =\cO (t^{-3/2}),
  \quad t\to\infty
\ee
\end{lemma}
\begin{proof}
Extending $F$ by $F(\om)=0$ for $\om<a$ and for $\om>b$,
we obtain a continuous function $F$ on $(-\infty,\infty)$
with $F'\in L^1(-\infty,\infty; {\bf B})$.
Using Zygmund's trick \ci[formula (4.2) p. 45]{Zyg}, we obtain
$$
\int\limits_{-\infty}^{\infty}F'(\om)e^{-it\om}d\om
=-\frac 12 \int\limits_{-\infty}^{\infty}(F'(\om+\frac{\pi}t)-F'(\om))
e^{-it\om}d\om
$$
Furthermore, the conditions (\re{Zc}) imply that
\beqn\nonumber
&&\int\limits_{-\infty}^{\infty}\Vert F'(\om+\frac{\pi}t)-F'(\om)\Vert d\om
=\int\limits_{-\infty}^{a+\pi/t}...+ \int\limits_{a+\pi/t}^{\infty}...
\\
\nonumber
&&\le 2\!\!\int\limits_a^{a+2\pi/t}\Vert F'(\om)\Vert d\om
+\!\!\int\limits_{a+\pi/t}^{\infty}\!\!d\om\!\!\int\limits_{\om}^{\om+\pi/t}
\Vert F''(\nu)\Vert d\nu=\cO (t^{-1/2})
+\frac{\pi}t\!\!\int\limits_{a+\pi/t}^{\infty}\Vert F''(\nu)\Vert d\nu
=\cO (t^{-1/2})
\eeqn
Hence, (\re{Zyg}) follows.
\end{proof}

\end{document}